%% file: main.tex
\documentclass[ amsfonts,reqno]{article}
\input{preamble}

\title {Hill's equation,  tire tracks and rolling cones}

\author{Gil Bor\footnote{
CIMAT, A.P.~402,
Guanjuato, Gto. 36000,
Mexico;
gil@cimat.mx} 
\and
Mark Levi\footnote{
Department of Mathematics,
Penn State,
University Park, PA 16802, USA;
levi@math.psu.edu}
}
\date{November 12, 2019}
\begin{document}
\maketitle

\begin{abstract} Louis Poinsot has shown in 1854 that the motion of a rigid body, with one of its points fixed,   can be described as the rolling without slipping  of one cone,  the `body cone', along another, the `space cone', with their common vertex at the fixed point. This description has been further refined by the second author in 1996, relating the geodesic curvatures of the spherical curves formed by intersecting the cones with the  unit sphere in Euclidean $\R^3$, thus enabling a reconstruction of  the  motion of the body from knowledge of the space cone  together with the (time dependent) magnitude of the angular velocity vector.  In this article  we show that a   similar  description exists for  a time dependent family of unimodular  $ 2 \times 2 $ matrices in terms of rolling cones  in  3-dimensional Minkowski space $\R^{2,1}$ and  the associated `pseudo spherical' curves,  in either the   hyperbolic plane $H^2$ or its Lorentzian    analog $H^{1,1}$.  In particular, this yields   an apparently new geometric interpretation of 
Schr\"odinger's (or Hill's) equation 
$ \ddot x + q(t) x =0 $  in terms of rolling without slipping of curves in the  hyperbolic plane.
\end{abstract}

\tableofcontents

\section{Introduction} The motion of a rigid body in $\R^3$,  with one of its points fixed,  consists at every moment  of rotation about an instantaneous  axis passing through the fixed point, also called the {\em angular velocity axis}. This is well known and easy to imagine (see for example the book \cite[p.~125]{Ar}). What is perhaps less well known is the following  remarkable  19th century theorem of Louis Poinsot \cite{Po}, describing the motion   in terms of {\em rolling without slipping of one cone along another:} 
\begin{quote}
{\em When a body is continuously moving round one of its points, which is fixed, the locus
of the instantaneous axis in the body is a cone, whose vertex is at the fixed point: the
locus of the instantaneous axis in space is also a cone whose vertex is at the fixed point
[\ldots] the actual motion of the body can be obtained by making the former of these
cones (supposed to be rigidly connected with the body) roll on the latter cone (supposed to
be fixed in space).} 
(Quoted from \cite[p.~2]{Wh}). See Figure \ref{fig:poinsot}. 
\end{quote}

\begin{figure}[h!]
    \def\svgwidth{\textwidth}
    \import{./}{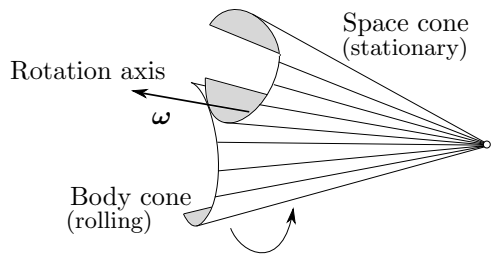}

\caption{\small Poinsot's Theorem: the body cone is rolling without slipping on  the space cone, and is tangent to it  along the instantaneous axis of rotation.}\label{fig:poinsot}
\end{figure}

As the second author has shown \cite{Ma}, this rolling cones   description can be made more precise: if we intersect each of the cones in Poinsot's theorem with a sphere centered at the fixed point we obtain a pair of spherical curves whose geodesic curvatures are related by the magnitude of the angular velocity vector $\ob$, enabling a reconstruction of  the  motion of the body from knowledge of the space cone  together with the (time dependent) magnitude $|\ob|$ (see Theorem \ref{thm:so3} below for the precise statement). 

Poinsot's Theorem can be reformulated  more abstractly as a statement about smooth curves in the orthogonal group $\SO_3$. It is natural to look for an analog  for other groups. In this paper we do that for the M\"obius group $\mathrm{PSL}_2(\R)\simeq \SO_{2,1}$. Poinsot's Theorem and its refinement of \cite{Ma} then  become a statement about the phase flow of the non-autonomous Hamiltonian linear system of ordinary differential equations
\be\label{eq:sys}\dot\x(t)=a(t)\x(t),
\ee
where $\x(t)\in\R^2$ and $a=a(t)\in\slt$, the space of $2\times 2$ traceless matrices. The salient features  of this interpretation are: 

\begin{itemize}
\item Solving equation \eqref{eq:sys} is equivalent to reconstructing a curve on a `pseudo-sphere' in Minkowski's space  $\R^{2,1}$ from its geodesic curvature.
\item  The  phase flow of \eqref{eq:sys}  can be 
visualized as a rigid motion in  $\R^{2,1}$, under which motion one
cone rolls on another without slipping.
\item  The rigid motion, and thus the solutions to  equation\eqref{eq:sys}, is completely determined by two cones, the `body cone' and the `space cone',  lying in    $\R^{2,1}$ and
 given explicitly in terms of  $a(t)$.

\item Unless $a(t)$  is a commuting family of matrices, the system  \eqref{eq:sys} cannot be solved explicitly by the na\"ive formula $\x(t)=\exp\left(\int_0^t a(\tau)\d \tau\right)\x(0)$ (unlike in the scalar version of this equation). Nevertheless,  the rolling cones interpretation  allows for a correction of this   formula in terms of parallel transport along curves in the pseudo-sphere in $\R^{2,1}.$ Interestingly, the cumulative angle of rotation appears in the solution despite the fact that the $a(t)$ do not commute. 
\end{itemize}

\mn  {\bf Plan  of the paper}.  In the next section, Section \ref{sec:back}, we describe in more detail Poinsot's Theorem and its refinement due to \cite{Ma}, see Theorem \ref{thm:so3}. 
In Section \ref{sec:sl2}  we formulate our main result, Theorem \ref{thm:sl2}, generalizing Theorem \ref{thm:so3} to rigid motions in Minkowski's space, thus giving a novel `rolling cones' interpretation  to the phase flow of system 
\eqref{eq:sys}. Section \ref{sec:proofs} contains a proof  of both Theorem \ref{thm:so3} and \ref{thm:sl2} in a unified group theoretic language, so as to make the generalization  from   $\SOt$ to $\SLt$  straightforward, see Theorem \ref{thm:both}. 
In the last two sections,  we  illustrate  our main result  via two examples of equation \eqref{eq:sys}: periodically perturbed harmonic oscillator (Mathieu's equation) and  the 2D bicycling equation.

\section{Background}\label{sec:back}

Consider the motion of a rigid body in Euclidean $\R^3$,  with one of its points fixed at the origin. 
If we follow any of the points of the body, initially at $\x(0)$, then its position   $\x(t)\in\R^3$  at time $t$ satisfies 
$$
\dot\x(t)=\ob(t)\times\x(t),
$$
 where  $\ob(t)\in \R^3$  is the associated {\em angular velocity vector} -- a vector aligned with the axis of rotation, whose  length $|\ob(t)|$ is the angular velocity of the body  about the axis of rotation and whose   direction is given  by the `right hand rule'. 
 
 Denote by $a_\ob :\R^3\to\R^3$
the map $\x\mapsto \ob\times \x$; then  the  last equation can be rewritten as  the  non-autonomous linear system
\be\label{eq:rigbod}
\dot\x(t)=a(t) \x(t),\qquad \mbox{where }\x(t)\in\R^3, \ a(t)=a_{\ob(t)}\in\sot, 
\ee 
and where $\sot$ denotes the space of $3\times 3$ antisymmetric real matrices. An  equation equivalent to   \eqref{eq:rigbod}  is the  equation for its fundamental solution matrix $g(t)\in\SO_3$
(the group of $3\times 3$ orthogonal matrices with determinant 1), satisfying 
\be
\dot g(t)=a(t)g(t), \ g(0)=\II,  
\qquad  \mbox{ where } \ g(t)\in\SO_3,\ a(t)=a_{\ob(t)}\in \sot,
\label{eq:fund}
\ee
and   $\II$ denotes  the identity $3\times 3$ matrix. The relation between the solutions of equations \eqref{eq:rigbod} and \eqref{eq:fund}  is $\x(t)=g(t)\x(0).$

Figure \ref{fig:conesintro} illustrates the above mentioned Poinsot theorem and the geometrical solution of equation \eqref{eq:fund}. 
In the figure, $\Cs$ denotes the locus of   rotation axes of the body, the `space cone'  (the cone, with vertex at the origin, generated by the space curve $\ob(t)$). Viewed from a body-fixed frame, the  rotation axes form another cone, the `body cone'  $\Cb$,  rigidly attached to the body, with vertex at the origin as well. Then,  as the body moves according to equation \eqref{eq:fund}, the cone $\Cb$ (rigidly affixed to the body) rolls without slipping along $\Cs$: 
at each moment, $\Cb$  is {\em tangent} to $\Cs$ along the instantaneous  axis  of rotation, which is (momentarily) at rest.

As shown in \cite{Ma}, this rolling cones   description can be made more precise, as follows. For a given non-vanishing `space angular velocity' curve $\ob(t)$ and a solution $g(t)$ to equation \eqref{eq:fund}, let $\Ob(t)=g(t)^{-1}\ob(t)$ be the `body angular velocity' curve, and    $\spa(t):= \ob(t)/|\ob(t)|, \bod(t):= \Ob(t)/|\Ob(t)|$ the (parametrized)  intersections of $\Cs, \Cb$ (respectively) with the unit sphere $S^2\subset \R^3$. 

\begin{theorem}[\cite{Ma}]\label{thm:so3}
\begin{enumerate}[(1)]

\item $g(t)$ rolls $\bod$ without slipping along $\spa$; that is: $g(t)\bod(t)=\spa(t)$, $g(t)\dot\bod(t)=\dot\spa(t)$, for all $t$. See Figure \ref{fig:conesintro}.

\item For non vanishing $\dot \spa$, the (spherical) geodesic curvatures $K,k$ of $\bod, \spa$ (respectively) are related by
\be\label{eq:geod1}
K=k-{|\ob|\over |\dot \spa|}.
\ee

\item Let $R[\Phi(t)]$ be  the rotation about  $\ob(0)$  by the angle $\Phi(t)=\int_0^t |\ob(\tau)|\d\tau$.  Then 
\be\label{eq:decomp}
g(t)=P_\spa(t)\circ R[\Phi(t)]\circ P_\bod(t)^{-1},
\ee
where $P_\bod(t)$ is (spherical) parallel transport along $\bod$ from $\bod(0)$ to $\bod(t)$,
 extended to $\R^3$ by $\bod(0)\mapsto \bod(t)$ and  similarly for $P_\spa(t).$
  
\end{enumerate}
\end{theorem}

Statement (1) is just a reformulation of Poinsot Theorem. Statement  (2), taken together with statement (1),  can be thought of as a geometrical/mechanical  `recipe' for
 solving equation \eqref{eq:fund}: given a `space angular velocity curve' $\ob(t)$, one uses  equation \eqref{eq:geod1} to construct $\bod(t)$ from its geodesic
  curvature and  the initial conditions $\bod(0)=\spa(0)$, $\dot\bod(0)=\dot\spa(0)$. Then $g(t)\in \SO_3$ is the (unique)
  rigid motion mapping   $\bod(t)\mapsto \spa(t),$ $\dot\bod(t)\mapsto  \dot\spa(t)$. 
  
  Statement (3) of Theorem \ref{thm:so3} is a curious fact 
regarding    `composition of a non-commuting  family of matrices'. Namely, the difficulty of solving   \eqref{eq:fund} explicitly 
   lies in the fact that,  in general,  the matrices $ a(t) $ do   not commute for different values of $t$. If, on the other hand,  the  axis of rotation is fixed, i.e., $\ob(t)=\omega(t)\e$ for some fixed  unit vector $\e$ and a scalar function $\omega(t)$, so that the $a(t)$ commute,
     then  $g(t)$ is the rotation about $\e$  by the cumulative angle 
     $\int_{0}^{t}\omega( \tau ) \d\tau ,$ i.e.,  
$g(t)=\exp\left(\int_{0}^{t}a( \tau ) \d\tau \right)$ 
is the solution to equation \eqref{eq:fund}, just as in the scalar version of equation \eqref{eq:fund}. In spite of the  lack of commutativity in general, the cumulative angle still appears in the decomposition formula \eqref{eq:decomp}, with an appropriate  correction by  parallel translations.

\begin{figure}
    \def\svgwidth{\textwidth}
    \import{./}{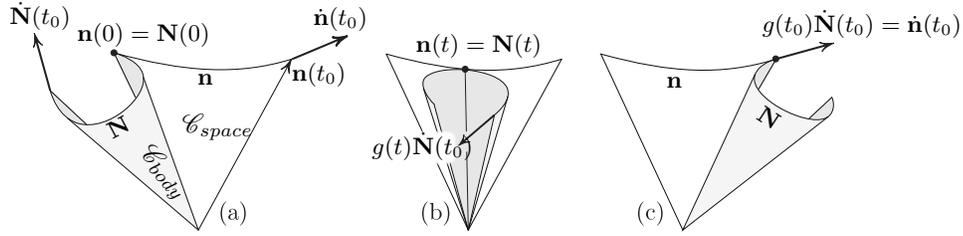}

	\caption{\small A  view of the cone $\Cb$  rolling along the   cone $\Cs$ without slipping under the rigid motion  $g(t)$. The 	curves ${\bf N}$, ${\bf n}$  are the intersections of these cones with the unit sphere.}
	\label{fig:conesintro}
\end{figure}

Here is a  heuristic explanation for the decomposition formula  \eqref{eq:decomp}.     
As the body curve $\bod$ rolls along $\spa$  in some time range $0\leq t\leq t_0$, the vector $ \dot {\bf N} (t_0) $ in 
Figure~\ref{fig:conesintro}  swings over and coincides with $ \dot {\bf n} (t_0) $ at $ t=t_0 $. The first key idea is that  {\it  this hard-to-describe motion can  be decomposed into two simpler ones}, as shown in Figure~\ref{fig:decomp}:   tangent transport $\TT^{-1} _{\bf N}$ of  $ \dot {\bf N} (t_0) $   along $  {\bf N} $ backwards to   
$ {\bf N} (0)= {\bf n} (0) $, followed by  tangent transport $\TT_ {\bf n}$ forward along $ {\bf n} $   to $ {\bf n} (t_0) $: 
\begin{equation} 
	\dot {\bf n} (t_0)= (\TT_ {\bf n}\circ \TT^{-1} _{\bf N} )\, \dot  {\bf N} (t_0).
	\label{eq:tt}
\end{equation} 
But 
\[
	\TT_ {\bf n} = P_ {\bf n} \circ R(\theta _ {\bf n}), \  \  \TT_ {\bf N} = P_ {\bf N} \circ R(\theta _ {\bf N}), 
\]  
where $P_ {\bf n} $ denotes   parallel transport along ${\bf n}$, $\theta _ {\bf n}$ is the integral of the geodesic curvature of ${\bf n}$ and $R(\theta) $ is the rotation around $ {\bf n}(0) = {\bf N} (0)$ through the angle $\theta$; thus (\ref{eq:tt}) becomes   
\begin{equation} 
		\dot {\bf n} (t_0) = g(t_0)\,\dot  {\bf N} (t_0)= (P_ {\bf n} \circ R(\theta _{\bf n}-\theta _{\bf N})\circ P^{-1}_ {\bf N})\, \dot  {\bf N} (t_0). 
		\label{eq:compintro}
\end{equation}  

The second key idea is the observation that {\it the  angle $ \theta _{\bf n}-\theta _{\bf N} $ turns out  to be the  time integral of the angular velocity $|\ob(t)|$ of the rigid motion  $g(t)$} 
 -- this is made precise by  equation  \eqref{eq:geod1}, relating the geodesic curvatures of ${\bf N}$ 
and of ${\bf n}$.    
\begin{figure}
    \def\svgwidth{\textwidth}
    \import{./}{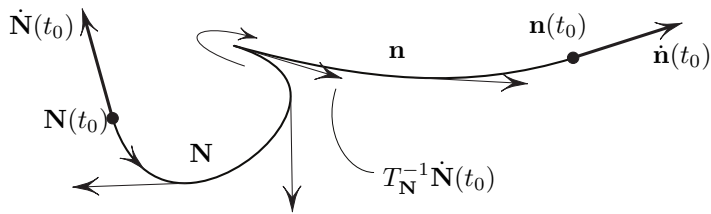}

\caption{\small The map  $ \dot {\bf N} (t_0)\mapsto\dot {\bf n} (t_0) $ is a composition of tangent transport backwards along ${\bf N}$ and 	forward along ${\bf n}$. This composition can be accomplished instead by parallel transport backwards along ${\bf N}$, followed by a 
	rotation around the cusp point, followed by parallel transport forward along ${\bf n}$. The angle of the rotation around the cusp turns out to be the 
	integral of the angular velocity of the rigid motions $ g(t)\in \SO_3 $.  } 
	\label{fig:decomp}
\end{figure}

\section{The main result }\label{sec:sl2} We apply the above ideas to gain geometrical insight into the linear system of ordinary differential equations
\be \label{eq:symp}
 \dot\x(t)=a(t)\x(t),\qquad \mbox{ where }\x(t)\in\R^2, \ a(t)\in\slt,
\ee
and  where $\slt$ denotes the set of traceless $2\times 2$ matrices.  
 This  system  includes,  among numerous  applications in mathematics,  physics and engineering, the  1-dimensional Schr\"odinger's, or Hill's, equation 
\be \label{eq:hill}
\ddot x  + q (t)x = 0,
\ee 
where  $x=x(t)$  and $q(t)$ are real  functions. The last equation is  
obtained as a special case of \eqref{eq:symp} by setting 
$$\x(t)={x(t)\choose \dot x(t)}, \quad a(t)=\left(\begin{matrix}0& 1\\-q(t)&0\end{matrix}\right).
$$
Another special case of \eqref{eq:symp}  is the `planar bicycle equation' (see Section \ref{sect:be} below). 

The fundamental solution matrix $g$ of  \eqref{eq:symp}, defined (as before) by
\be\label{eq:lie1}
\dot g(t)=a(t)g(t), \quad  g(0)=\II, 
\ee
lies in  $\SLt$,  the group of $2\times 2$ matrices with determinant 1. As before,  the relation between the solutions of equations \eqref{eq:symp} and \eqref{eq:lie1}  is $\x(t)=g(t)\x(0).$

The starting point of our approach is the observation that the    linear area--preserving flow in $\R^2$  of  equation \eqref{eq:symp} 
{\it can  equivalently be viewed as a rigid motion in the Lie algebra } 
$\slt$.  More precisely, instead of considering the motion of points in $\R^2 $ under $g\in\SLt$, we consider  the motion of points in $ \slt$,  the $3$--dimensional Lie algebra of $\SLt$, given by conjugation with $g$:    
$$
	\Ad_g:  \sl_2( \R)\rightarrow \sl_2( \R)  ,\quad a\mapsto gag^{-1}, \quad a\in\slt, \quad g\in \SLt.
$$
	 Now $ \Ad_g $, being a conjugation,  preserves the spectrum of each $a\in\slt$, and in particular, 
$\det(a)$. Since $\tr(a) = 0 $, $\det(a)$   turns out to be an indefinite quadratic form, which  makes $\slt$  a Minkowski space (we provide the details later in Section \ref{sec:prelim}). Thus, {\em $ \Ad_g$ is an   orthogonal   transformation  of the Minkowski space $\slt\simeq\R^{2,1}$}, a `rigid motion'. 
The map $ g\mapsto \Ad_g $ is $ 2 $ to $1$, so up to a minor ambiguity, all  properties of $g$ can be recovered from those of 
$ \Ad_g $. For instance, $g$ is elliptic, i.e., conjugate to a  rotation of $ \R^2 $  through an angle $\theta$, if and only if  $ \Ad_g $
 is a rigid rotation in $\slt$ (in the Minkowski metric)  around a timelike axis, rotating the  orthogonal (spacelike) plane  through the angle $ 2 \theta $; 
 similar statements hold for  parabolic  and hyperbolic elements in $\SLt$. 

One  advantage of looking at $ \Ad_g $ acting on $ \sl_2( \R)$ (versus $ g $ acting on $ \R^2 $) is that a  geometry (hidden heretofore in $\R^2$) is revealed;   the already mentioned orthogonality of $ \Ad_g $ is one example. Furthermore,  orthogonal transformations of  Minkowski's space, just like Euclidean ones,   have axes of rotation:  lightlike  for the  elliptic rotations  and spacelike  for the hyperbolic ones; in $\R^2$, none of this is visible. 

 By carrying through  this analogy  between Euclidean and Minkowski rigid motions, we then obtain,  with some minor modifications due to sign and nullity  details,   the following    almost-verbatim Minkowski version of  Theorem 
  \ref{thm:so3}.

\begin{theorem}\label{thm:sl2}
Let $a(t)\in\slt$  be a given non-vanishing `space angular velocity' curve  with non vanishing $|a|:=2\sqrt{|\det(a)|} $ and   let $g(t)\in\SLt$ be   the 
 solution to  
$\dot g=ag,$ $g(0)=\II$. Let  
 $A=g^{-1}a g$  be 
  the associated `body angular velocity' curve and    $\spa:= a/|a|, \bod:= A/|A|$ be the projections of $a,A$ (respectively) on the unit `pseudo-sphere'  $\Sigma\subset \slt$ (either the hyperbolic plane $H^2$ or its Lorentzian analog $H^{1,1}$, depending on the  sign of $\det(a)$; see Section \ref{sec:prelim} below for  details). Then 
\begin{enumerate}[(1)]

\item $g(t)$ rolls $\bod$ without slipping along $\spa$, i.e.,  $\Ad_{g(t)}\bod(t)=\spa(t)$, $\Ad_{g(t)}\dot\bod(t)=\dot\spa(t)$, for all $t$.

\item For non vanishing $|\dot \spa|$, the (pseudo-spherical) geodesic curvatures $K,k$ of $\bod, \spa$ (respectively) are related by
$$K=k-{|\ob|\over |\dot \spa|}.
$$

\item Let $R[\Phi(t)]$ be  the (pseudo) rotation about  $a(0)$  by the angle $\Phi(t)=\int_0^t |\ob(\tau)|\d\tau$.  Then 
$$\Ad_{g(t)}=P_\spa(t)\circ R[\Phi(t)]\circ P_\bod(t)^{-1},$$
where $P_\bod(t)$ is parallel transport along $\bod$ from $\bod(0)$ to $\bod(t)$,
 extended to $\slt$ by $\bod(0)\mapsto \bod(t)$ and  similarly for $P_\spa(t).$
  
\end{enumerate}

\begin{rmrk}In the above theorem, the assumption that $|a(t)|$ is non-vanishing, i.e., the space angular velocity is nowhere null, is essential. For the special
case of Hill's equation \eqref{eq:hill}, this amounts to assuming that the potential $q(t)$ does not vanish for all $t$.
Studying this case of $a$ crossing the null cone remains an interesting question which we do not
address in this paper.
\end{rmrk}

\end{theorem}

\section{Notation and  setup} \label{sec:proofs}
 We start with a review of some notation and terminology, mostly standard.  

\subsection{Geometry and algebra of $\SOt$ and $\SLt$}\label{sec:prelim}

Denote in the following by $G$ either $\SO_3$ or $\SLt$ and by $\g$ its Lie algebra,  either $\so_3$  or $\slt$, respectively. 
The conjugation action of $G$ on $\g$, $\Ad:G\to \mathrm{GL}(\g),$  is denoted by 
\be
	\Ad_g(a)=g\cdot a:=gag^{-1}, \quad g\in G,\ a\in\g.
	\label{eq:cong}
\ee 

Define an Ad-invariant inner product   on $\g$ by   
\be\label{eq:inner}
\<a,b\>:=\lambda\,\tr(ab),
\mbox{ where $\lambda=-{1\over 2}$ for $\g=\so_3$  and $\lambda=2$ for $\g=\slt.$}
\ee 
Our choice of the normalization factor for each $\g$ will be explained in a moment. In either case, we set $$|a|:=\sqrt{|\<a,a\>|}.$$ 
The $\Ad$-invariance of $\<\ ,\ \>$  implies that $b\mapsto [a,b]=ab-ba$ is an anti symmetric operator on $\g$ with respect to $\<\ ,\ \>$, i.e., $\<[a,b],c\>=-\<b, [a,c]\>$ for all $a,b,c\in\g$,  hence 
\be\label{eq:as}
\<[a,b],a \>=0,\quad \forall a,b\in\g.\ee
 
\mn 

 Let us examine the resulting geometry of $\g$ in each of the two cases. 
 
 \bn{\bf Case 1:} $\g=\so_3$.
With the choice $\lambda=-{1\over 2}$ in \eqref{eq:inner}, $\<a,b\>:=-\tr(ab)/2$ is a  positive  definite inner product on $\so_3$, the image of the standard inner product on $\R^3$ under the isomorphism  $\R^3\to \sot$, $\o\mapsto a_{\ob} \in  \so_3$, where $a_{\ob}\x:=\ob\times\x.$ Explicitly, 
\be\label{eq:aw}
\ob=\left( \begin{array}{c}\omega _1 \\ \omega _2  \\  \omega _3\end{array} \right)
\mapsto 
a_{\ob}=	\left( \begin{array}{ccc} 0 & -\omega _3 & \omega _2\\ \omega _3 &  0 & -\omega _1  \\  -\omega _2 &  \omega _1 & 0 
	\end{array} \right). 	
\ee
Furthermore, under this isomorphism, the cross  product $\u\times\v$ corresponds to the Lie bracket $[a,b]=ab-ba$ and  the standard action of $\SO_3$ on $\R^3$ corresponds to the conjugation action   (\ref{eq:cong}); that is, 
$$\<a_\u, a_\v\>=\u\cdot\v,\ [a_\u, a_\v]=a_{\u\times \v},  \ g\cdot a_\u=a_{g\u}, \quad \mbox{for  }   g\in\SO_3,\ \u,\v\in\R^3.$$

 \bn{\bf Case 2:} $\g=\slt$. The Lie algebra $\slt$ consists of  traceless $2\times 2$ real matrices, which we choose to write in the form 
$$
	a={1\over 2}\left(\begin{matrix} a_1& a_2+a_3\\a_2- a_3&-a_1\end{matrix}\right), 
$$
so that $ \<a,b\>:=2\tr(ab)=a_1b_1+a_2b_2-a_3b_3.$
Thus  the inner product  is indefinite,  of signature $++-$ (the `spacelike sign convention'). A simpler formula for the associated quadratic form   is
$$\<a,a\>=(a_1)^2+(a_2)^2-(a_3)^2=-4\det(a), \quad a\in\slt.$$

An element $a\in\slt$ is called {\em timelike} if $\<a,a\><0,$ {\em lightlike} (or {\em null}) if $\<a,a\>=0$ and {\em spacelike} if  $\<a,a\>>0.$ These are the three {\em causal types} of elements in  $\slt$,  also referred to as {\em elliptic, parabolic} and {\em hyperbolic}, respectively.

\sn 

The reason  for our  choice   $\lambda=2$ in  formula \eqref{eq:inner} for $\g=\slt$  is the following analog of a familiar property of the vector product in $\R^3$.

\begin{lemma}\label{lemma:bracket} If $a,b\in\slt$ is an orthonormal pair, i.e., $|a|=|b|=1$ and $\<a,b\>=0$, then $(a,b,[a,b])$ is an orthonormal frame in $\slt$, positively oriented with respect to the standard volume form $a_1\wedge a_2\wedge a_3$ if $a,b$ are spacelike, and negatively oriented if one of them is timelike. 
\end{lemma}

\begin{proof}

It is easy to check that   
\be\label{eq:basis}
{\bf i} :=   \frac{1}{2} \left( \begin{array}{cc} 1 &0 \\ 0 & -1 \end{array} \right), \quad
{\bf j} :=   \frac{1}{2} \left( \begin{array}{cc} 0 &1 \\ 1 & 0 \end{array} \right),
\quad
	{\bf k} :=   \frac{1}{2} \left( \begin{array}{cc} 0 &1 \\ -1 & 0 \end{array} \right)
\ee
is an  orthonormal basis of $\slt$,   dual to $a_1, a_2,a_3 $, hence it  is positively oriented with respect to $a_1\wedge a_2\wedge a_3.$ Furthermore,  $\bi, \bj$  are spacelike and $\bk$ is timelike,  satisfying 
\be\label{eq:comm}
	[ {\bf i}, {\bf j} ] = {\bf k}, \ \  [{\bf j}, {\bf k} ] = -{\bf i} , \ \ [{\bf k} , {\bf i} ] = -{\bf j}. 
\ee

Now let  $a,b\in\slt$ be an orthonormal pair. Since $a,b$ are  not null and  orthogonal, both are spacelike or  one is timelike and the other spacelike.
In the first case, where $a,b$ are spacelike orthogonal unit vectors, by  conjugating by an appropriate element of $\SLt$ and (possibly) permuting them (neither operation  changes the orientation of $(a,b,[a,b])$), we can assume that $a=\bi$, $b=\bj$, thus $[a,b]=\bk$, hence $(a,b,[a,b])$ is a positively oriented orthonormal frame.

In the second case, where one of $a,b$ is timelike and the other spacelike, 
 by (possibly) permuting $a$ and $b$ and changing $a$ to  $-a$ (these operations do not affect the orientation of $(a,b,[a,b])$), we can assume that $a$ is timelike future pointing ($a_3>0$) and $b$ is  spacelike. Next,  by  conjugating by an appropriate element of $\SLt$,  we can assume  that $a=\bk$ and $b=\bi$, so that  $[a,b]=-\bj$, and hence  
$(a,b,[a,b])$ is a negatively oriented  orthonormal frame, as claimed.
 \end{proof}

\begin{rmrk}
The commutation relations \eqref{eq:comm}
differ from the analogous relations for the cross product in $ \R^3 $ by the ``$-$" sign when the timelike 
vector ${\bf k}$ occurs  in the commutator. Putting it differently, when taking the cross product in the Minkowski space $\slt$, one uses the `right-hand rule' to  determine the direction of the cross product of two spacelike vectors,  and  the `left-hand 
rule' whenever a timelike vector participates in the cross product. 
\end{rmrk}

\subsection{Rolling without slipping}

Denote by $\Sigma\subset \g$ the unit (pseudo) sphere, i.e., the  set of elements $a\in \g$ with $\<a,a\>=\pm 1.$ Thus, for $\g=\so_3$, $\Sigma$ is the standard 2-sphere $S^2=\{a\in\so_3\st \<a,a\>=1\}$, while for $\g=\slt$, $\Sigma$ is either $H^{2}:=\{a\in\slt\st \<a,a\>=-1\}$ (hyperboloid of two sheets), or $H^{1,1}:=\{a\in\slt\st \<a,a\>=1\}$ (hyperboloid of one sheet), see Figure \ref{fig:mink}.

\begin{figure}
    \def\svgwidth{\textwidth}
    \import{./}{fig4.tex_pdf}

\caption{Level sets of  $\<a , a \>=(a_1)^2+(a_2)^2-(a_3)^2$   in $\slt$.}\label{fig:mink}
\end{figure}

\mn 
 
Now let $g(t)$ be a smoothly parametrized curve in  $G$ with $g(0)=\II$ (the identity element in $G$). Define  
\be\label{eq:ang}A(t):=g^{-1}(t)\,\dot g(t), \quad  a(t):=\dot g(t)\,g^{-1}(t)\in\g,
\ee  the {\em  body} and {\em space angular velocities}, respectively, and 
\be\label{eq:angn}\bod(t):=A(t)/|A(t)|, \quad \spa(t) :=a(t)/|a(t)|,\ee 
the radial projections of $A(t),a(t)$ (respectively) onto $\Sigma\subset\g$. Note that in order to define the (pseudo) spherical curves $\bod(t), \spa(t),$ we need to assume that $|a(t)|\neq 0$ for  all $t$, which we assume henceforth. For $G=\SO_3$ this amounts to $a(t)\neq 0$;  for $G=\SLt$ it means that $a(t)$ is non null for all $t$, i.e., it is either spacelike or timelike.

 From equations  \eqref{eq:cong} and \eqref{eq:ang}, we  have
\be\label{eq:one}
\dot g=ag=gA, \ a=g\cdot A, \ \spa=g\cdot\bod ,\   g(0)=\II.
\ee

\begin{rmrk}[About notation] Sometimes,  as in \eqref{eq:one}, we  suppress the explicit dependence on $t$,  i.e.,  $g=g(t) ,$  $a=a(t) $, etc. 
\end{rmrk}

\begin{definition}[Rolling without slipping]\label{def:roll}
 Let $\Gamma(t),\gamma(t)$ be two parametrized    curves in  $\g$. A {\em rolling  without slipping of $\Gamma$ along $\gamma$} is a parametrized curve  $g(t)$ in $G$,  satisfying  for all $t$ the {\em contact} and {\em no slip} conditions: 
\begin{align}\label{eq:roll1}
g(t)\cdot \Gamma(t)&=\gamma(t), \\
 g(t)\cdot \dot \Gamma(t)&=\dot \gamma(t).\label{eq:roll2}
\end{align}
See Figure \ref{fig:rolling}. 
\end{definition}

\begin{figure}
    \def\svgwidth{\textwidth}
    \import{./}{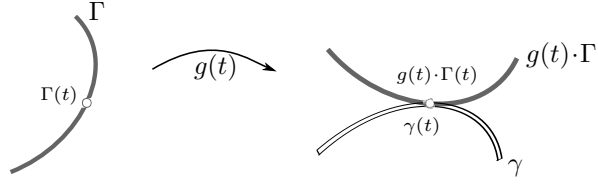}

	\caption{Rolling of a   curve $\Gamma$  along $\gamma$ via  a family of isometries $\Ad_{g(t)}.$} 
	\label{fig:rolling}
\end{figure}

\begin{lemma}
The no-slip condition \eqref{eq:roll2} is equivalent to 
\be\label{eq:roll3}
[a,\gamma]=0,
\ee
where $a=\dot g g^{-1}.$ This expresses the vanishing of the velocity of the `material point' of the moving curve at the contact point $g(t)\cdot \Gamma(t)$ between the two curves. 
\end{lemma}
\begin{proof}Taking the derivative with respect to $t$ of  equation~\eqref{eq:roll1} and using equations \eqref{eq:one}, 
$$[a,\gamma]+g\cdot \dot\Gamma=\dot \gamma.$$
Thus $g\cdot \dot \Gamma=\dot \gamma$ (equation \eqref{eq:roll2}) is equivalent to $[a,\gamma]=0$. 
\end{proof}

\subsection{Geodesic curvature}

 Let $\gamma(t)$ be a smoothly parametrized (pseudo) spherical curve in $ \Sigma\subset\g$ with  nowhere null tangent, i.e., $|\dot\gamma|$ does not vanish, and let $\gamma':=\dot\gamma/|\dot\gamma|$ be the unit tangent along $\gamma$. Then $(\gamma, \gamma', [\gamma,\gamma'])$ is a  `moving'  orthonormal frame along $\gamma$.  

\mn {\bf Notation.} We denote henceforth by dot  derivative along a curve $\gamma$ with respect to an arbitrary parameter $t$, $\dot\gamma:=\d\gamma/\d t, $  and by prime  derivative with respect to arc length parameter $s$, $\gamma':=\d \gamma/\d s=\dot\gamma/|\dot\gamma| $  (provided $|\dot\gamma|$ does not vanish).

\begin{definition}\label{def:geod}The {\em geodesic curvature} of an oriented  (pseudo) spherical curve $\gamma$ in $\Sigma \subset \g$ with nowhere null  tangent   is  its {\em normal acceleration}, i.e., the coefficient of $[\gamma, \gamma']$  in the decomposition 
 of $\gamma''$ as a linear combination of $\gamma,\gamma', [\gamma, \gamma']$. 
 
\end{definition} 

This definition can be also expressed conveniently as  
\be \label{eq:geod}
\gamma''\equiv k[\gamma, \gamma']\  \mod \gamma,\gamma'.
\ee
For an arbitrary  parametrization $\gamma(t)$,  $\gamma'' \equiv \ddot \gamma/|\dot\gamma|^2\equiv k[\gamma, \gamma'] \ \mod \gamma,\gamma'$, from which  follows 
$$\ddot \gamma
\equiv  k |\dot\gamma|[\gamma,\dot\gamma]\ \  \mod \gamma, \dot\gamma.$$

\begin{rmrk}[About the sign of the geodesic curvature] Our Definition \ref{def:geod} of  geodesic curvature  may differ in sign from other common definitions in the literature, since this sign depends on the choice of a unit normal to the curve. Our choice of unit normal  $[\gamma,\gamma']$ is mostly for simplicity in subsequent formulas. At any rate, all applications of this definition  in this article are invariant under   sign change of $k$. For example, equation \eqref{eq:curv} below. 
\end{rmrk}

\subsection{Parallel transport} A vector field $v(t)$ tangent to $\Sigma$ along $\gamma(t)$ is {\em parallel} if $\dot v(t)\perp T_{\gamma(t)}\Sigma$ for all $t$. That is,
$$\dot v\equiv 0 \quad  \mod \gamma.$$
Any initial vector $v(0)\in  T_{\gamma(0)}\Sigma$ can be extended uniquely  to  parallel vector field $v(t)$ along $\gamma$, by solving the last displayed equation (a linear system of ODEs). The resulting map $P_\gamma(t): T_{\gamma(0)}\Sigma\to T_{\gamma(t)}\Sigma$, $v(0)\mapsto v(t)$, is an isometry (with respect to  the restriction of $ \<\ ,\ \>$ to  $\Sigma$),  called {\em parallel transport} along $\gamma$. 

\sn 

 The two notions, geodesic curvature and parallel transport, are related as follows. Let $\gamma(t)$ be a (pseudo) spherical curve with non vanishing $|\dot\gamma|$  and  $v(t)$   the parallel transport of $\gamma'(0)$ along $\gamma$ (or any parallel vector field along $\gamma$ with  the same causal type  as $\gamma'$). At each point $\gamma(t)$ along the curve, $\gamma'$ is related to $v$  by a unique  orientation preserving  isometry  $R(\theta)$ of  $T_{\gamma(t)}\Sigma$,    with `rotation angle' $\theta$. That is, in the Riemannian case, 
\be\label{eq:rot1}
\gamma' = R(\theta)v=(\cos\theta)v+(\sin \theta)[\gamma,v], 
\quad \Sigma=S^2 \mbox{ or } H^2,
\ee
  and in the Lorentzian case 
\be\label{eq:rot2}
\gamma' = R(\theta)v=(\cosh\theta)v+(\sinh \theta)[\gamma,v], \quad \Sigma=H^{1,1}.
\ee  
   
   \begin{lemma}\label{lemma:para} For any oriented  curve $\gamma$ in $\Sigma$ with non-null tangent, its geodesic curvature  is the rate of change, with respect to arc length,  of the `rotation angle'  of the unit tangent $\gamma'$, relative to  a parallel unit vector of  the same causal type as $\gamma'$, as defined in equations \eqref{eq:rot1}-\eqref{eq:rot2}; that is, 
   $$k=\theta'.$$ 
It follows that 
$$\gamma'(t)=R\left[\theta(t)\right]P_\gamma(t)\gamma'(0)=P_\gamma(t)R\left[\theta(t)\right]\gamma'(0),$$
 where $$\theta(t)=\int_0^{L_t}k\,\d s=\int_0^tk|\dot\gamma|\d \tau,$$ 
and where $s$ is an arc length parameter along $\gamma$, $L_t$ is the length of $\gamma$ between $\gamma(0)$ and $\gamma(t)$ and $\tau$ is the same parameter as $t$. 
\end{lemma}

\begin{proof}From $\gamma'=R(\theta)v$  follows, by a simple calculation,  $\gamma''=\theta'(\partial_\theta R(\theta))v+R(\theta)v'\equiv\theta'[\gamma,R(\theta)v]= \theta'[\gamma, \gamma']\  \mod \gamma$, implying $k=\theta'$. 
\end{proof}

\begin{rmrk} In case $\Sigma=H^{1,1}$,  $|\dot\gamma(t_0)|$ may vanish even if  $\dot\gamma(t_0)\neq 0.$  Then one cannot reparametrize $\gamma$ by arc length and $k$ becomes infinite at $t=t_0$. It would be interesting to understand the significance of this phenomena for a linear system  $\dot g=ag$. 
\end{rmrk}

\section{The combined  theorem and its proof}

 With the above background we now state and prove the following  result, which  combines Theorems \ref{thm:so3} and \ref{thm:sl2}. 
\begin{theorem}\label{thm:both} Let $G$ be either $\SLt$ or $\SOt$, $\g$ its Lie algebra, $a(t)$ a smoothly parametrized curve in $\g$  with non-vanishing $|\dot a|$, and $g(t)\in G$ the solution to  $ \dot g = a g $, $ g(0)=\II$.  Set $A(t)=g^{-1}(t)a(t),$ and $\bod(t), \spa(t)$   the  corresponding normalized (pseudo) spherical curves in $\Sigma\subset\g$, as defined in equations~\eqref{eq:ang}--\eqref{eq:angn}.
Then  

\begin{enumerate}[(1)]
\item  (Poinsot Theorem) $g(t)$ rolls    without slipping the curve $A(t)$ along $a(t)$ and $\bod(t)$ along $\spa(t)$. 

\item (The reconstruction formula)  If $|\dot \spa|$ is non-vanishing then the geodesic curvatures $K, k$ of the (pseudo) spherical curves $\bod, \spa$  (respectively) are related  by 
\be\label{eq:curv}
K=k-{|a|\over  |\dot\spa|}.\ee

\item (The decomposition formula)
$$\Ad_{g(t)}=\tilde P_\spa(t) \circ R\left[\Phi(t)\right]\circ (\tilde P_\bod(t))^{-1},$$
where $\tilde P_\spa(t)$ is  parallel transport $T_{\spa(0)}\Sigma\to T_{\spa(t)}\Sigma$ along $\spa$, extended to $\g$ by $\spa(0)\mapsto \spa(t)$,  similarly for $\tilde P_\bod(t)$, and  $R\left[\Phi(t)\right]$
is the (pseudo) rotation around the axis $a(0)$ by the angle $\Phi(t)=\int_0^t|a(\tau)|\d \tau.$
\end{enumerate}
\end{theorem}

\begin{proof}  (1)  If  $\gamma=g\cdot \Gamma$ then $\dot \gamma=\dot g\cdot \Gamma+g\cdot \dot\Gamma=[a,\gamma]+g\cdot \dot \Gamma.$ For $\gamma=a, \Gamma=A$, since $a=g\cdot A$ and $[a,a]=0$, we get  $\dot a=g\cdot\dot A.$
Next, $\spa=g\cdot \bod$ implies  $\dot \spa=[a,\spa]+g\cdot\dot \bod=|a|^{-1}[a,a]+g\cdot\dot \bod=g\cdot\dot \bod$.

\mn (2)  Applying $g$ to  $\ddot \bod\equiv  K|\dot\bod|  [\bod, \dot\bod]\  (\mod \bod, \dot\bod)$, we obtain $g\cdot \ddot \bod\equiv  K|\dot\spa|  [\spa, \dot\spa]\  (\mod \spa, \dot\spa).$ Taking derivative of  $\dot \spa=g\cdot\dot \bod$, we get 
$\ddot \spa=[a,\dot\spa]+g\cdot\ddot \bod=|a|[\spa,\dot\spa]+g\cdot\ddot \bod\equiv (|a|+ K|\dot\spa| ) [\spa, \dot\spa]\  (\mod \spa, \dot\spa).$ On the other hand,  $\ddot \spa\equiv  k|\dot\spa|  [\spa, \dot\spa]\  (\mod \spa, \dot\spa),$ hence $|a|+ K|\dot\spa|= k|\dot\spa|,$ which gives formula \eqref{eq:curv}.
 
\mn (3) Both sides of the equation are  orientation preserving isometries of $\g$, mapping $\bod(t)\mapsto\spa(t)$, hence it is enough to show that they coincide on $\bod'(t)$. By Lemma 
\ref{lemma:para} and equation \eqref{eq:curv},  
$$
\spa'(t)=P_\spa(t)R\left[\theta(t)\right]\spa'(0), \qquad 
\bod'(t)=P_\bod(t)R\left[\Theta(t)\right]\bod'(0),$$
 where $\theta(t)=\int_0^{L_t} k\,\d s$ and 
 $$
 \Theta(t)=     \int_0^{L_t}K\, \d s=\int_0^{L_t}\left(k -{|a|\over |\dot\spa|}\right)\d s
 =\int_0^{L_t} k\, \d s-\int_0^t|a| \d \tau=\theta(t)-\Phi(t).$$  
 It follows that 
 \begin{align*}
P_\spa(t)  R\left[\Phi(t)\right] (  P_\bod(t))^{-1}\bod'(t)&=
  P_\spa(t)  R\left[\Phi(t)\right] R\left[\theta(t)\right]\bod'(0)=\\
&= P_\spa(t)  R\left[\Phi(t)+\Theta(t)\right]\bod'(0)=\\
&= 
 P_\spa(t)  R\left[\theta(t)\right]\spa'(0)=\spa'(t),
 \end{align*} as claimed. 
\end{proof}

\section{Example: the Mathieu equation (timelike angular velocity)}\label{sec:examples}
In this section we illustrate Theorem \ref{thm:both} for $G=\SLt$ with a well-known example. The {\em Mathieu equation}
\be\label{eq:harm}
\ddot x+\omega^2(1+\epsilon  \cos t)x=0   
\ee
 can be thought of as a model of small--amplitude oscillations of a  pendulum whose pivot oscillates sinusoidally in the vertical direction.  
This system arises in numerous other settings which we will not list here. 
 We  can rewrite Mathieu equation as a system 
$$\dot\x=a\x, \hbox{ where } \x={x\choose \dot x},\quad a=\left(\begin{matrix}0&1\\ -\omega^2(1+\epsilon  \cos t)&0\end{matrix}\right)\in\sl_2(\R),$$
with the fundamental matrix   $g(t)\in\SL_2(\R)$ defined by $\dot g=ag,$ $g(0)=\II.$ From now on we   assume that $ |\epsilon |< 1 $, 
so that  
$\< a, a\>=-4\det(a)=-4\omega^2(1+\epsilon  \cos t)<0$, and thus  $a(t)$ is timelike.  Since the diagonal entries of  $a$ vanish, $a$ is constrained to the plane $a_1=0$, and thus the space curve $\spa$ follows a geodesic segment  on $H^2$ (unless $ \epsilon =0 $, in which case $\spa$ is a point); 
 in particular,  $k=0$ for the geodesic curvature of the space curve.  
 From equation~\eqref{eq:curv}, we obtain the expression for the geodesic curvature  of the body curve $\bod$: 

$$
	K=-{|a|\over |\dot\spa|}=-\frac{4 \omega   (1+\epsilon  \cos t)^{3/2}}{\epsilon |\sin t|}.
$$

Thus $\bod(t)$ has cusps at $t=n\pi,$ $n\in\Z$, see Figure~\ref{fig:rollingmatt}.

\begin{figure}
\centering\begin{tabular}{ccc}
\includegraphics[width=.2\textwidth]{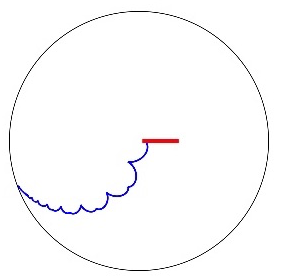}&
\includegraphics[width=.2\textwidth]{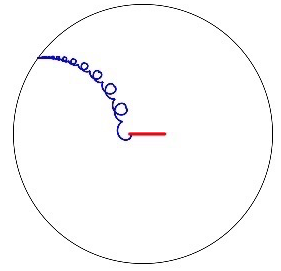}&
\includegraphics[width=.2\textwidth]{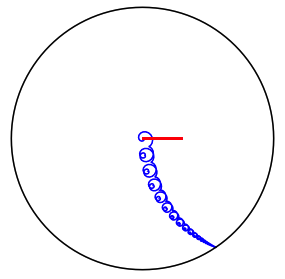}\\
$\omega=1/2$&$\omega=1$&$\omega=3/2$\\
&&\\
\includegraphics[width=.2\textwidth]{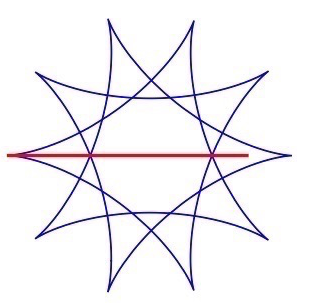}&
\includegraphics[width=.2\textwidth]{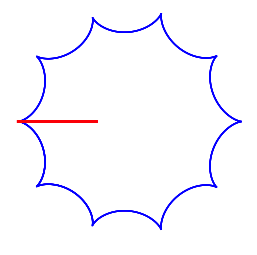}&
\includegraphics[width=.2\textwidth]{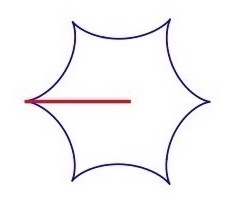}\\
$\omega\approx 1/5$&$\omega\approx 2/5$&$\omega\approx 1/3$
\end{tabular}
\caption{The Mathieu equation: The space curve $\spa$ (the horizontal  segment) and  the body  curve $\bod$ in the Poincar\'e disk model of $H^2$, for various choices of $ \omega $ and $\epsilon$. Top row: unstable case (hyperbolic period map); bottom row: stable case (elliptic period map).}
	\label{fig:rollingmatt}
\end{figure}

\mn 

We recall briefly that the  {\em period map} (the {\em monodromy}, or {\em Floquet matrix})  of equation \eqref{eq:harm} is defined by $M:=g(2\pi)\in \SLt$, where $g(t)$ is the fundamental solution of the associated linear system, and that 
it determines completely the stability properties of equation~\eqref{eq:harm} in the sense that all solutions are  bounded for all time if and only if $M$ is  
elliptic, or equivalently,  if and only  if the set of  its matrix powers $\{M^n | n\in\Z\}$ is  bounded.  Note that for $ | \epsilon | < 1 $, the infinitesimal generator $a(t)$ of the flow $g(t)$  of (\ref{eq:harm}),  for each  $t$,  is  elliptic, and yet   $M$, thought of as a composition of a non commuting family of infinitesimal elliptic rotations, may  itself fail to be elliptic, leading to unbounded solutions of   \eqref{eq:harm}, a phenomenon known as  {\em parametric resonance} \cite[\S 25, p.~113]{Ar}.  Figure~\ref{fig:tongues} shows the associated {\em Arnold tongues}: the shaded regions in the $(\omega, \epsilon ) $--plane, corresponding to the parameter values for which the period map $M$ is hyperbolic. 
\begin{figure}
\centering
\includegraphics[height=.23\textwidth]{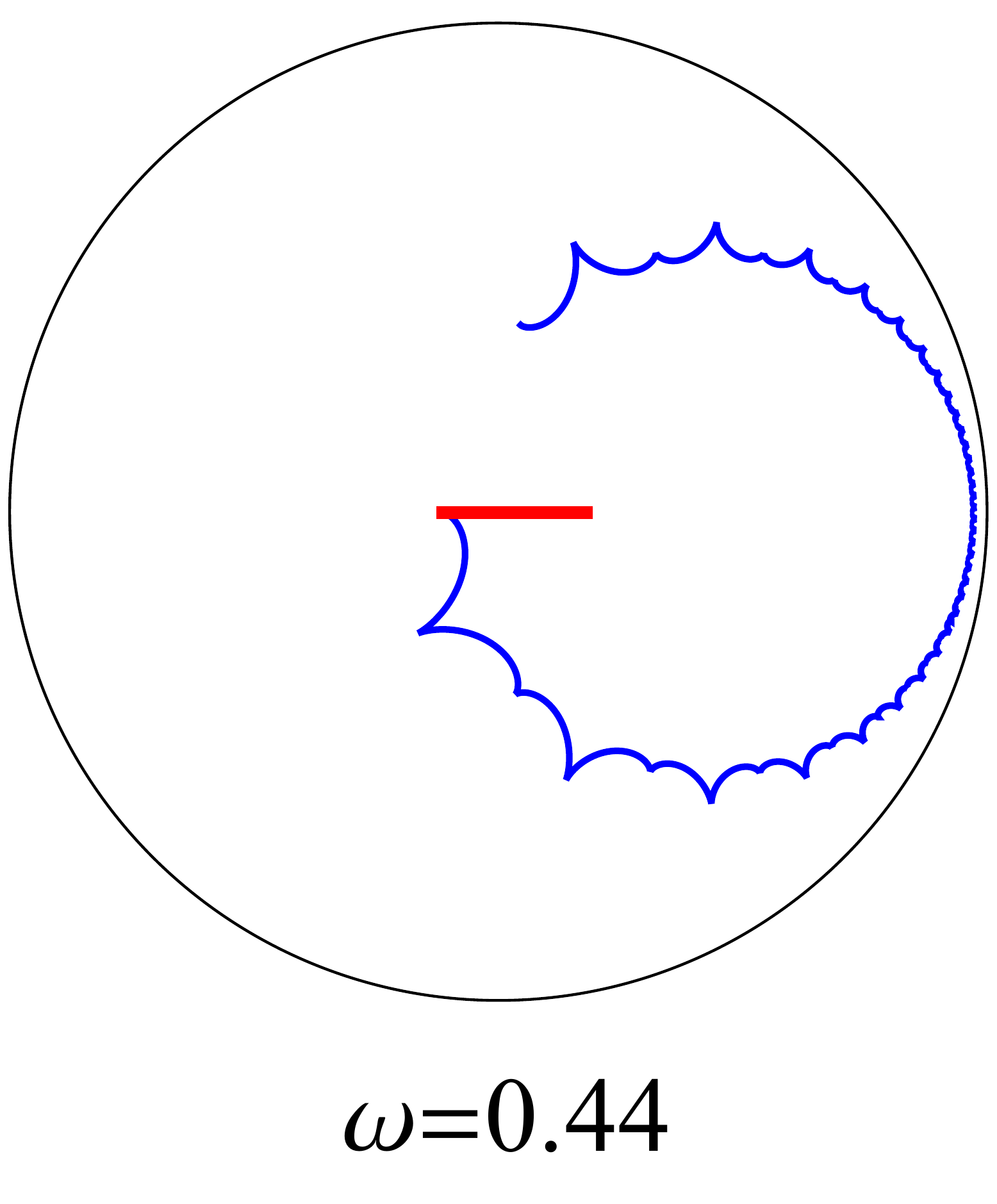}
\includegraphics[height=.23\textwidth]{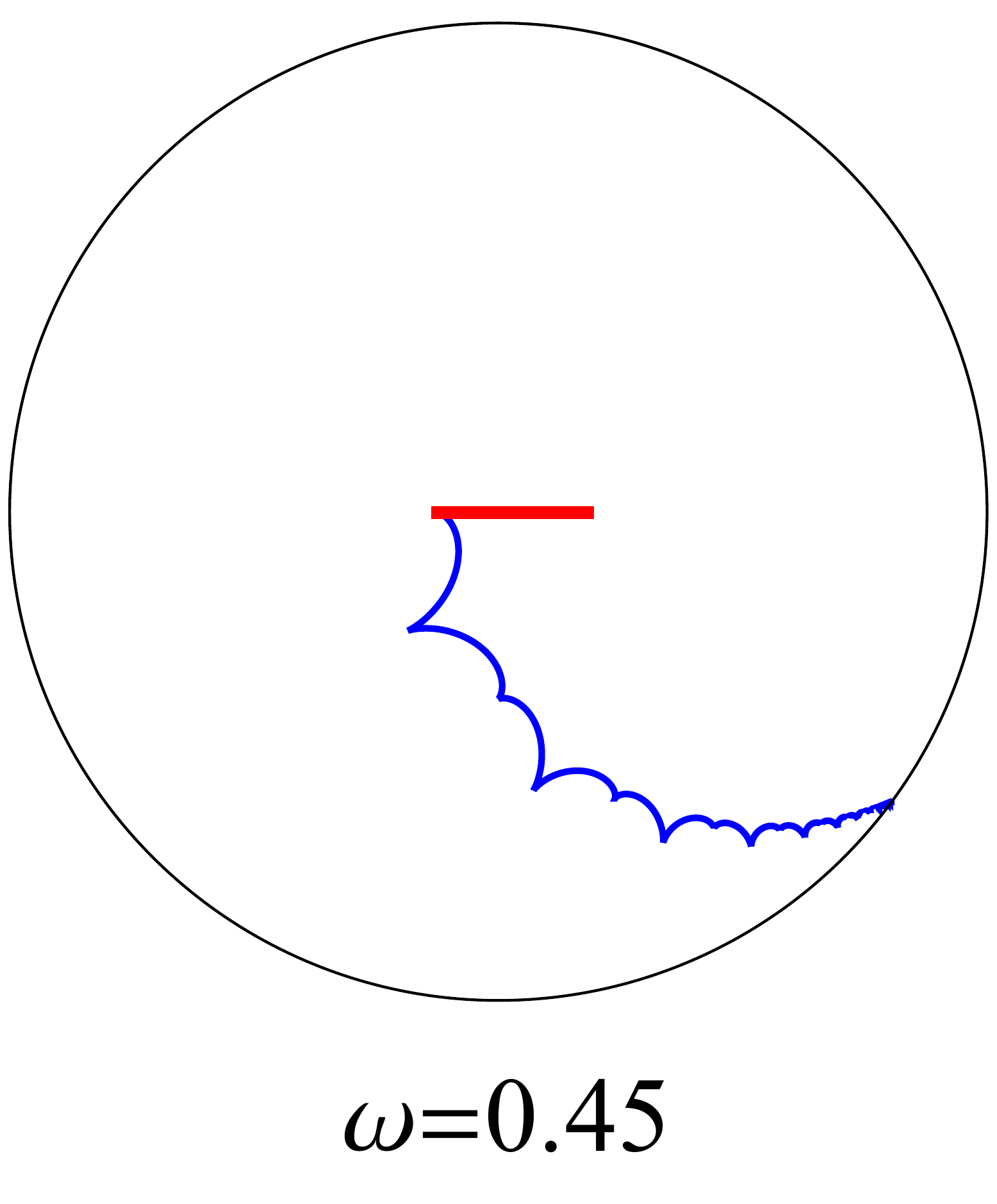}
\includegraphics[height=.23\textwidth]{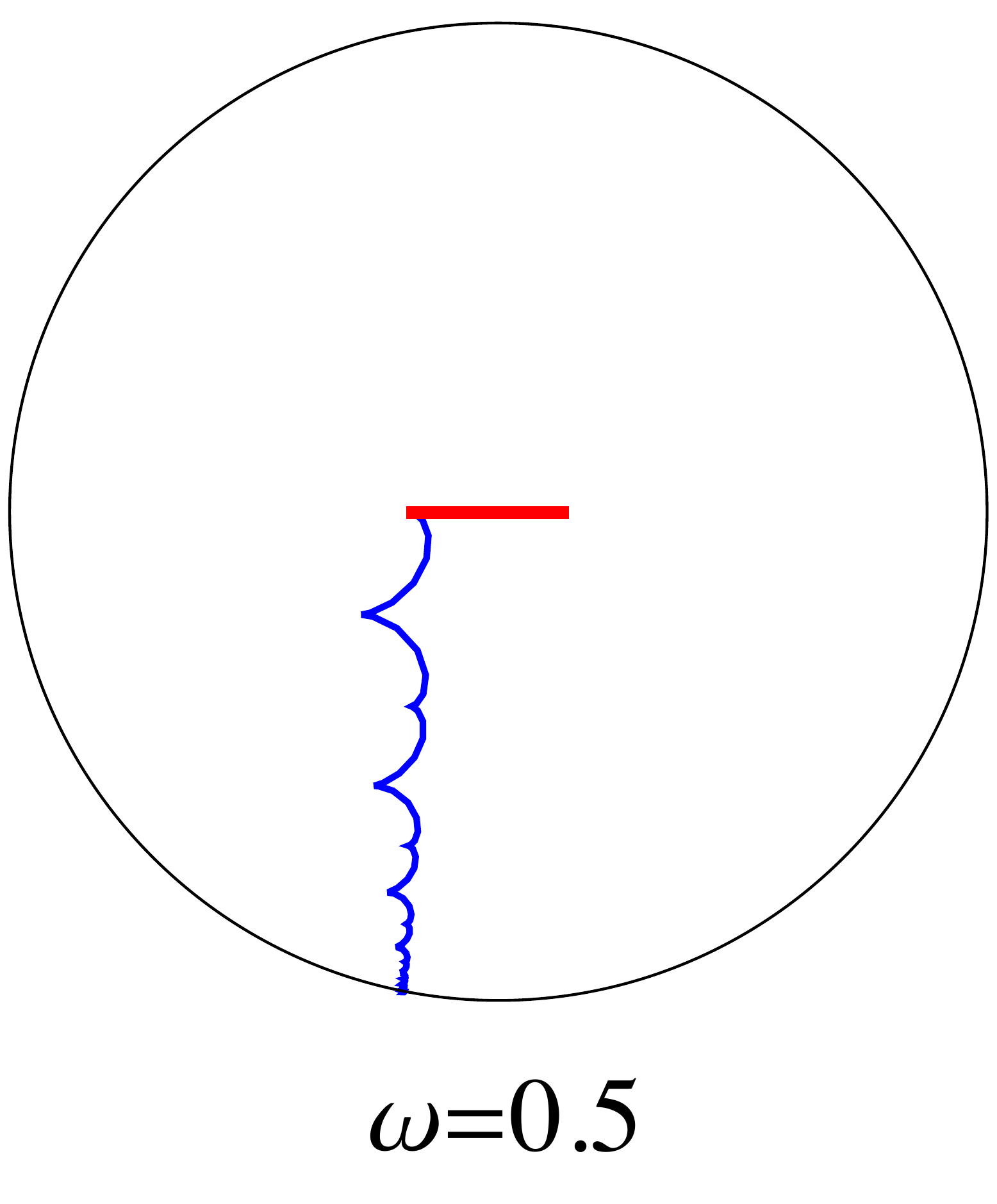}
\includegraphics[height=.23\textwidth]{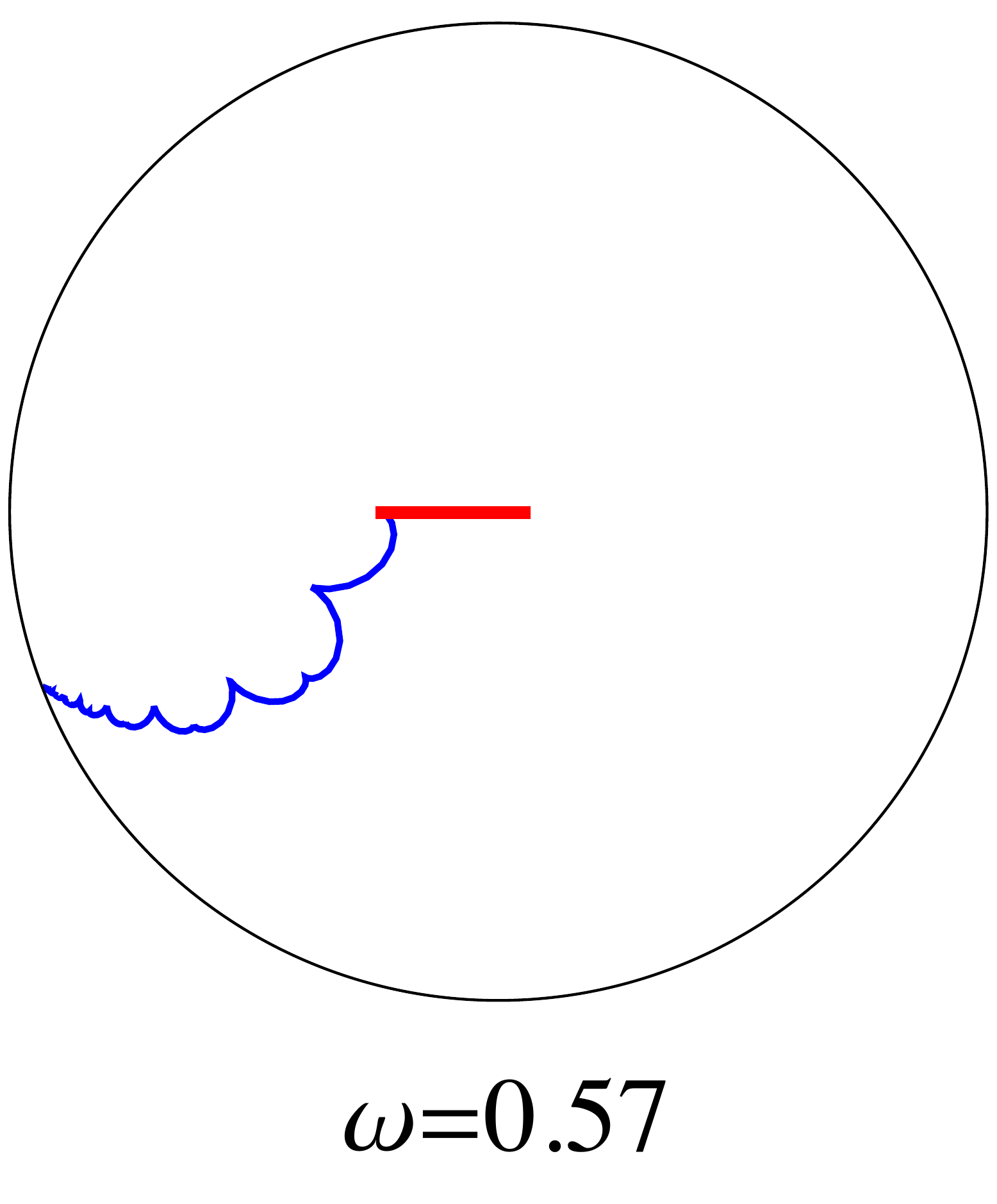}
\includegraphics[height=.23\textwidth]{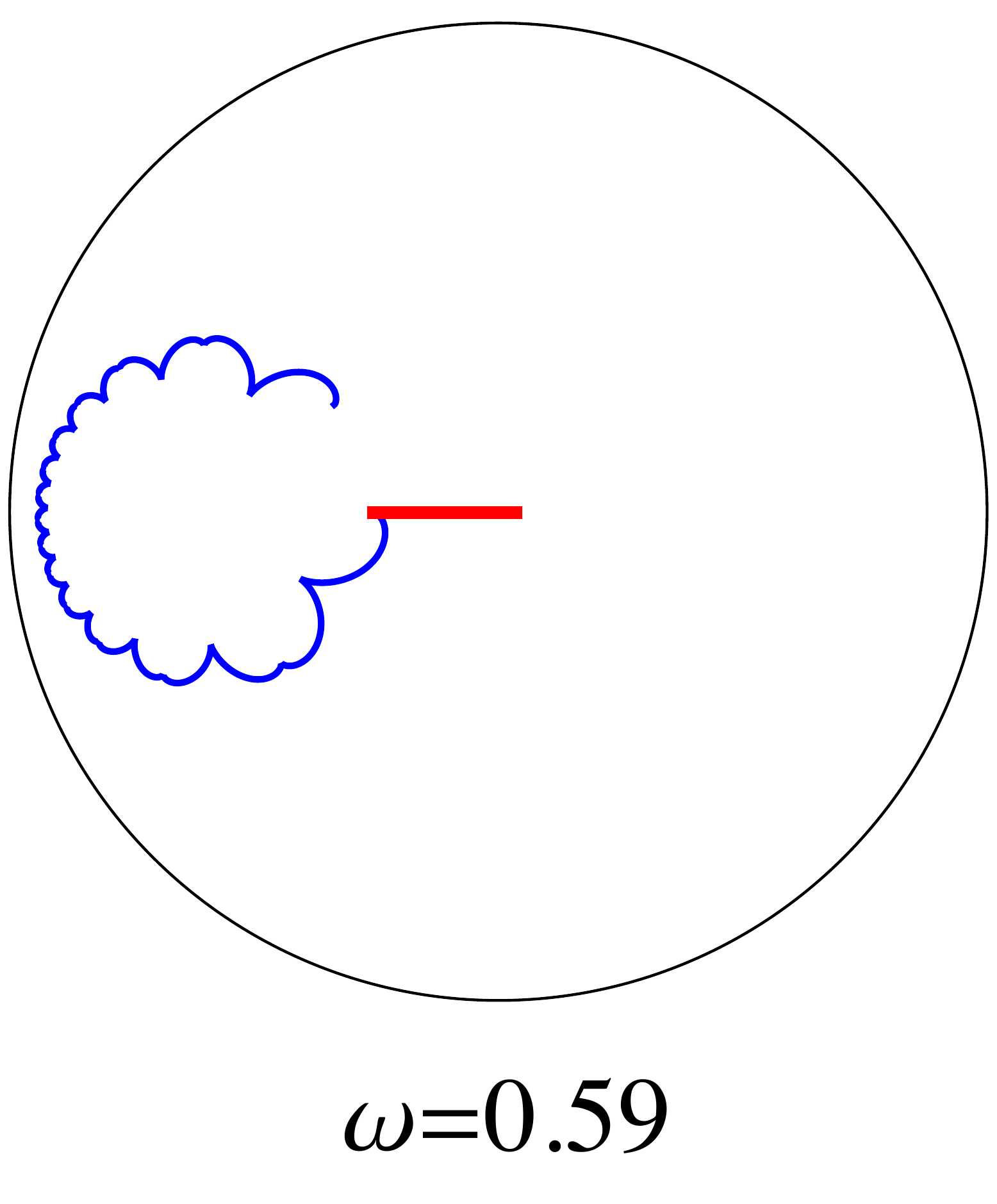}
\caption{As $ ( \omega, \epsilon )  $ crosses the first Arnold tongue of Figure~\ref{fig:tongues} (with fixed $ \epsilon = 0.55) $, the curve $ {\bf N} $  changes as shown, starting with the elliptic case  on the left, through hyperbolic (the middle three figures) and ending an elliptic monodromy again (right) in the next stability region.}
\label{fig:tonguecrossing}
\end{figure}

\begin{figure}
\centerline{\includegraphics[width=.6\textwidth]{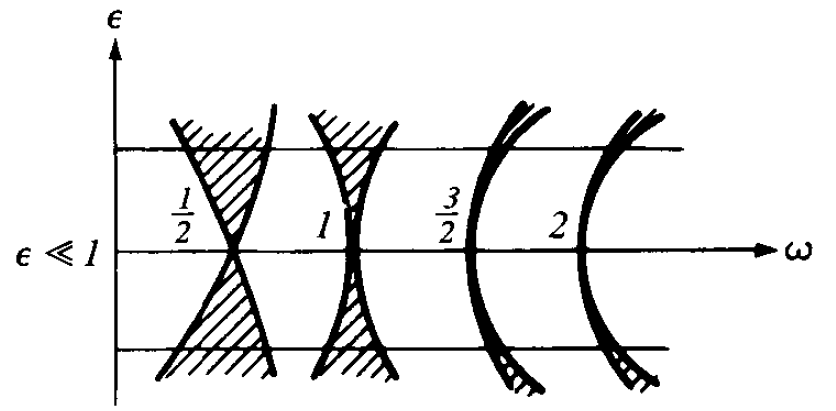}}
\caption{Arnold's tongues for the Mathieu equation.}
\label{fig:tongues}
\end{figure}

Returning to  the hyperbolic plane $H^2$, Figure~\ref{fig:tonguecrossing} illustrates how stability of the Mathieu equation is reflected in  the body curve $\bod(t)$: 
 for $ ( \omega ,\epsilon ) $ in the stable (unshaded)  region of Figure \ref{fig:tongues}, the body curve $ \bod $ is quasi-periodic or periodic, as must be the case 
since the set $\{M^n | n\in \Z\}$ is bounded. On the other hand, for all resonant $ ( \omega ,\epsilon ) $ (the shaded regions of Figure~\ref{fig:tongues}) the body curve $ \bod $ 
extends to the absolute (the `circle at infinity' in the Poincar\'e  disk model of $H^2$ in Figure ~\ref{fig:tonguecrossing}), reflecting the fact that the powers $ M^n $ are unbounded as $ | n | \rightarrow \infty $. 

We also point out that  if the period map  $M$ is elliptic, conjugate to a  rotation through an angle $ 2 \pi /n $, the body curve ${\bf N}$ is closed, with $ 2n $ cusps, as shown in the lower row of images in Figure \ref{fig:rollingmatt}.

 Figure~\ref{fig:rollingmatt} shows the `static' picture, i.e., the initial position of ${\bf N}$ at $t=0$; Figures~\ref{fig:snapshots} and \ref{fig:rollingmatt1} illustrate  the rolling of ${\bf N}$  on the space curve ${\bf n}$.   

  \begin{figure}
	\centerline{\includegraphics[width=.8\textwidth]{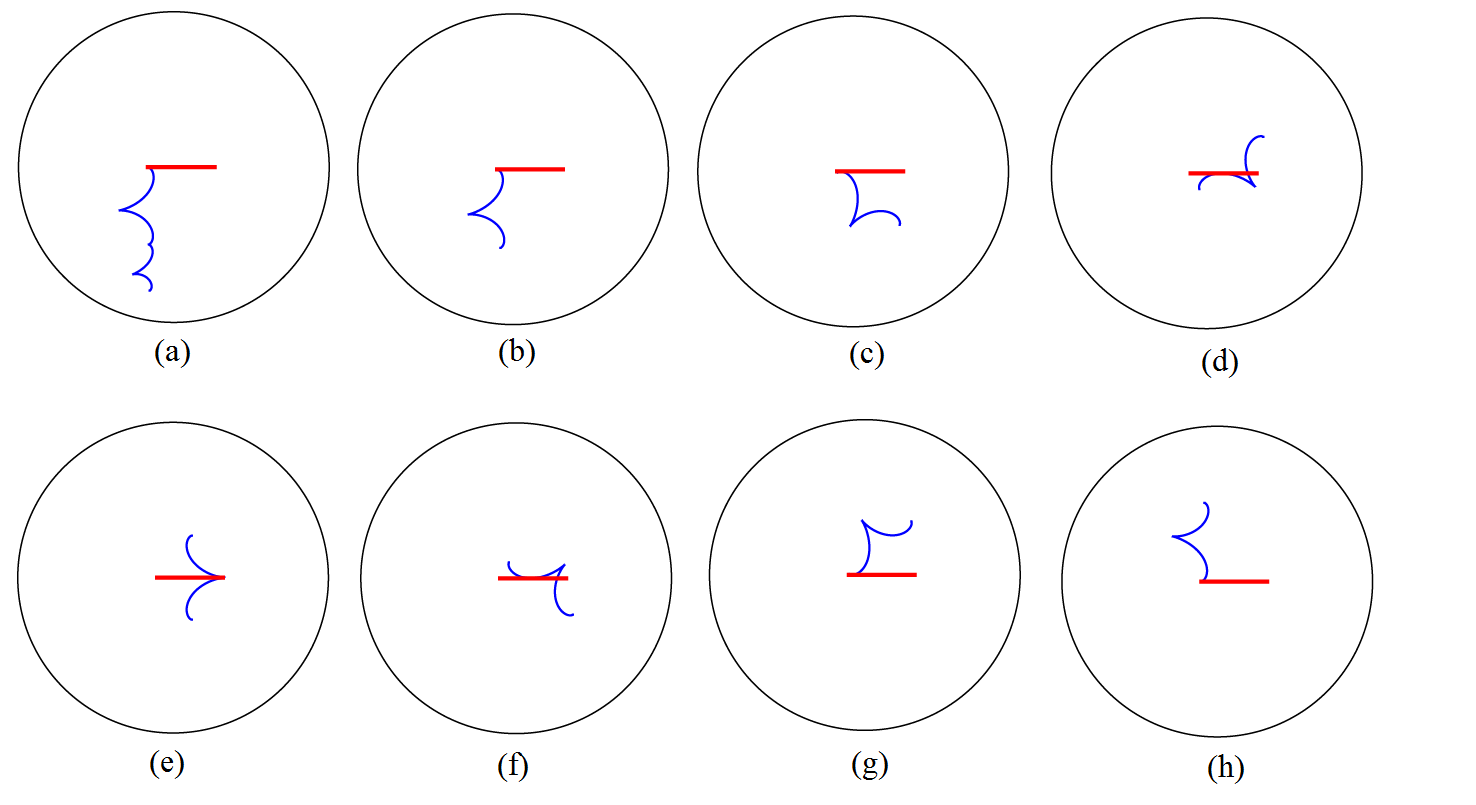}}
	\caption{\small (a): a piece of the `body' curve ${\bf N} $ and the space curve $\spa$ (the horizontal segment) in the Poincar\'e disk; 
	(b)-(h): some snapshots of a single `loop' of the body curve ${\bf N}$ rolling on the 
	space curve ${\bf n}$. }
	\label{fig:snapshots}
\end{figure}  
\begin{figure}
	\center{\includegraphics[width=0.3\textwidth]{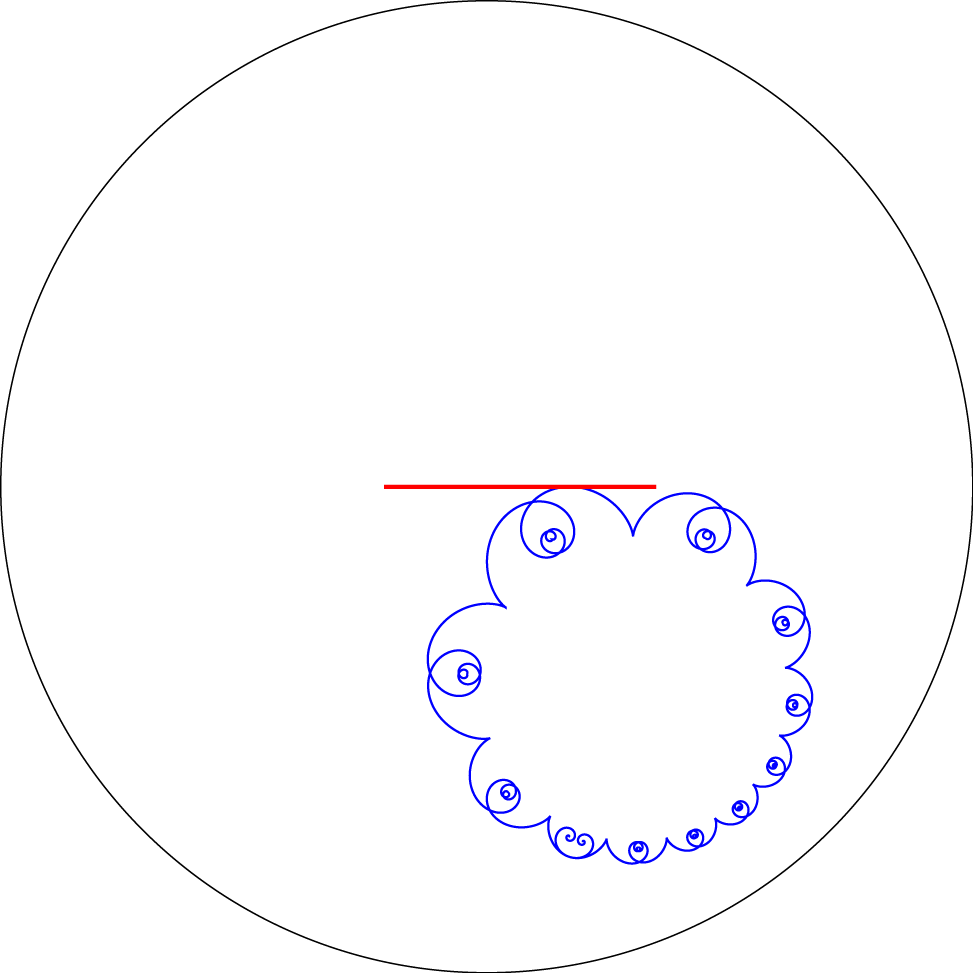}}
	\caption{\small Another elliptic case:   $ {\bf N} $ rolls on ${\bf n}$.}
	\label{fig:rollingmatt1}
\end{figure}

\section{Example: the bicycle equation (spacelike angular velocity)} \label{sect:be}

In this section we illustrate Theorem \ref{thm:both} for $G=\SLt$ with another example, where the motion of a `bicycle' is represented by rolling of cones in  Minkowski space; the bicycle is described in 
the caption of Figure~\ref{fig:bike}.

    \begin{figure}[thb]
    \def\svgwidth{\textwidth}
    \import{./}{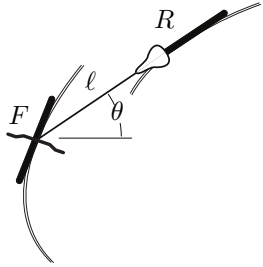}

	\caption{\small The `bicycle' is represented by  a segment $ RF $ of fixed length $ \ell $  whose `front end' $ F$ undergoes a prescribed motion along the `front track', and   
	whose `rear end' $R$  motion is constrained by the `no slip' condition: its velocity is aligned with  the segment $RF$ at all times. }
	\label{fig:bike}
\end{figure}  

We start by recalling the description the motion of a bicycle by a linear system of  ODEs. The `no slip' condition  is easily seen to be equivalent to the angle $\theta$ of the bicycle satisfying

\begin{equation} 
	\ell\dot \theta =\dot x \sin \theta-  \dot y \cos \theta , 
	\label{eq:bike}
\end{equation} 
where $F(t)= (x(t),y(t))$ is a parametrized  `front track'. Equation \eqref{eq:bike} is equivalent to 
\begin{equation} 
	\frac{d}{dt} \left(\begin{matrix}u\\v\end{matrix}\right)=-{1\over 2\ell}\left(\begin{matrix}\dot x&\ \ \dot y\\\dot y&-\dot x\end{matrix}\right)
	\left(\begin{matrix}u\\v\end{matrix}\right); 
	\label{eq:bikesys}
\end{equation}   
namely, for  any solution of the linear system  \eqref{eq:bikesys},  the angle
\begin{equation} 
	\theta = 2 \arg (u+iv)
	\label{eq:2arg}
\end{equation}  
evolves according to  equation   \eqref{eq:bike}. The proof of this equivalence is a straightforward calculation (see  \cite[Theorem 1]{BLPT}). 

The coefficients matrix $a(t)$ of the system  \eqref{eq:bikesys} satisfies
$
	\langle a, a \rangle =-  4\det(a) = (\dot x ^2 + \dot y  ^2 )/ \ell^2  > 0, 
$ 
so that $a$ is spacelike and $\spa=a/|a| \in H^{1,1}$.   From now on we   assume that the front track $F(t)$ is a closed convex  curve  of perimeter $L$, parametrized by arc length, i.e.,  $|\dot F|^2=\dot x^2+\dot y^2=1$, so   $\spa=-\dot x\,\bi-\dot y\,\bj$ is a  parametrization of the equator $x_3=0$  of $H^{1,1} $. In other words, the `space curve' follows  the equator;  in particular, the geodesic curvature of $\spa$ is $ k = 0 $. To calculate the geodesic curvature of the body curve  we use formula \eqref{eq:curv}, obtaining   $K=-|a|/|\dot\spa|=-1/(\ell\kappa)$, where $\kappa=|\ddot F|=\sqrt{\ddot x^2+\ddot y^2}$ is the curvature of the front track.  That is:  {\em the geodesic  curvatures of the body curve $\bod(t)\in H^{1,1}$ and the  front wheel track $F(t)\in \R^2$ are  reciprocal, up to a factor.} 

This surprising reciprocal connection between two curves living in different spaces -- the bike's front track in $\R^2$ and the body curve in $H^{1,1}$ -- was proven here by computation. It turns out, however, that there is a geometrical explanation of this reciprocity; we will provide this explanation elsewhere. 

\mn 

We now make some observations on the body curve. Since $ F(t) $ is assumed to be closed,    the coefficient matrix  of the bicycle system \eqref{eq:bikesys} is periodic; the Floquet matrix $M_\ell$  of this system is referred to as the {\em $\ell$-bicycle monodromy} of the front track. 
The monodromy $M_\ell$ may be  elliptic, parabolic or hyperbolic; as a side remark, in the latter case $M_\ell$ has two real eigendirections, which correspond to two    {\em closed} rear wheel tracks, as  Figure \ref{fig:bmonod} illustrates; one of these corresponds to the bike moving backwards. 

\begin{figure}
\centering%
    \def\svgwidth{\textwidth}
    \import{./}{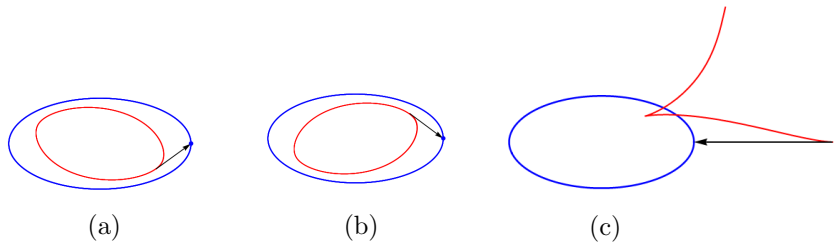}

\caption{\small Bicycle monodromy for an elliptical front track (blue): if the bicycle length $\ell$ is small enough the monodromy is hyperbolic; (a) and (b) show the two closed back tracks (red) corresponding to the two fixed point of $M_\ell$ in $\R P^1$. (c): for  $\ell$ large enough, the monodromy is elliptic, conjugate to a rotation.}\label{fig:bmonod}
\end{figure}

 \begin{figure}
\begin{tabular}{cccc}
\hspace{-.7cm}\includegraphics[width=.25\textwidth]{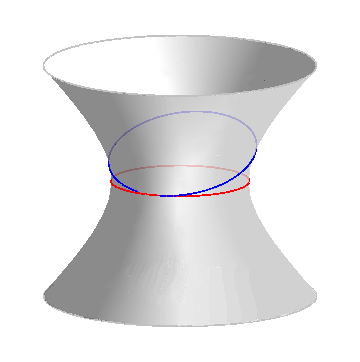}& 
\includegraphics[width=.25\textwidth]{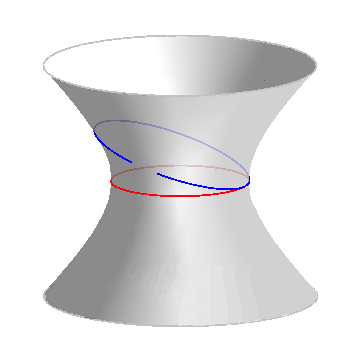}& 
\includegraphics[width=.25\textwidth]{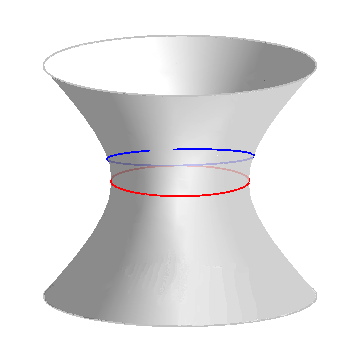}& 
\includegraphics[width=.25\textwidth]{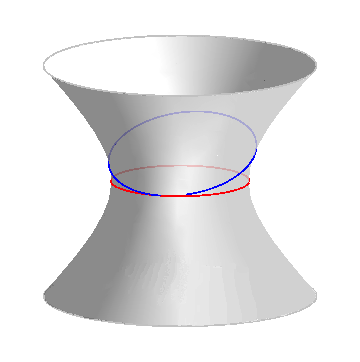}
\end{tabular}
\caption{\small Some snapshots of rolling curves in $H^{1,1}$, representing bicycling along a circular front track with elliptic monodromy ($\ell>$ radius of the front track).  The  `body curve' $\bfN$ (the tilted  ellipse) has  constant curvature $|K|>1$, and is rolling  
	along  the stationary `space curve' $\bfn$ (the `equator' of $H^{1,1}$, a geodesic). }
	\label{fig:rollingellipse}
\end{figure}

 \mn{\bf An example.} In the special case when the front track is the unit circle we have  $\kappa=1$,  $|K|=1/\ell$,  so $\bod$ is a spacelike constant
 geodesic curvature curve on $H^{1,1}$. Now all  curves of constant geodesic curvature on $H^{1,1}$ are given simply  by  plane
  sections of this hyperboloid (just like in case of the ordinary sphere ${ S}  ^2\subset\R^3$). In our case, the intersecting   plane is
   tangent to the equator  at $\spa(0)$, Figure~\ref{fig:rollingellipse}.  
   For $ \ell >1$ this plane section is an ellipse with geodesic curvature $|K|=1/\ell<1$, 
   as shown in  
Figure~\ref{fig:rollingellipse}, and the bicycle monodromy is elliptic. For $ \ell =1 $ the plane section is a parabola, with 
$| K| = 1 $ and $M_\ell$ parabolic.  Similarly, for $ \ell < 1 $ the plane section is a hyperbola, 
one branch of which is the body curve, with  asymptotes a pair of ruling null lines of $H^{1,1}$, with  $|K|=1/\ell>1$,  and the bicycle monodromy is hyperbolic.

\mn{\bf General closed front track.} In the general case when  $\kappa$ (the curvature of the bicycle front track $F$) is not constant and the bicycle length  $\ell$ is small enough, the bicycle monodromy 
$M_\ell$ is hyperbolic and the resulting body curve $\bod$ in $H^{1,1}$ is unbounded, asymptotic to one of the ruling null lines, as shown in Figure \ref{fig:hyproll}(b). For $\ell$ large  enough the bicycle monodromy is elliptic and the corresponding body curve is bounded quasi-periodic, filling up a `ribbon' wrapped around $H^{1,1}$, as illustrated in  Figure \ref{fig:hyproll}(d).
 \begin{figure}

\centering\begin{tabular}{ccc}
\centered{\includegraphics[width=.25\textwidth]{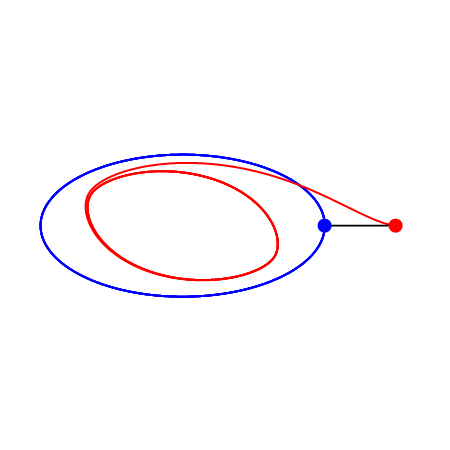}}
& \qquad  \qquad &
 \centered{\includegraphics[width=.3\textwidth]{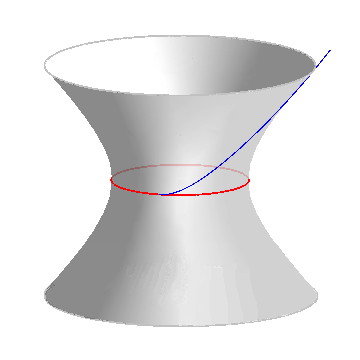}}\\
(a) && (b) \\
&& \\
\centered{\includegraphics[width=.3\textwidth]{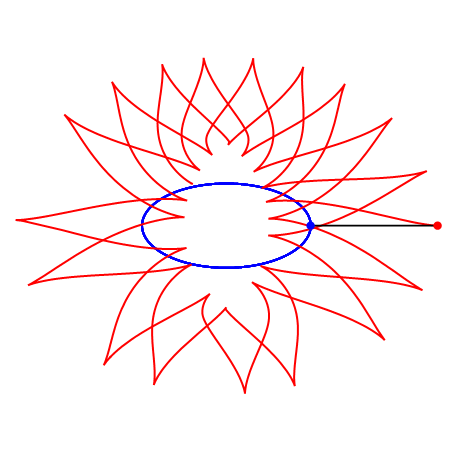}}
& \qquad  \qquad &
 \centered{\includegraphics[width=.3\textwidth]{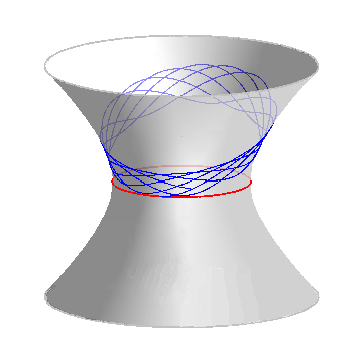}}
\\
(c) && (d) 
\end{tabular}
 \caption{\small (a) Bicycling along an elliptical front track (blue), with hyperbolic monodromy (small $\ell$). The rear track (red) spiral towards a closed curve, corresponding to the stable fixed point of the monodromy. 
	(b) The corresponding body curve (blue) is unbounded, asymptotic to one of the null lines on $H^{1,1}.$
	(c) An elliptical front track with elliptic monodromy. The rear track is quasi-periodic. (d) The corresponding body curve is contained  in  a `ribbon' wrapped  around $H^{1,1}.$
}
	\label{fig:hyproll}
\end{figure}

Returning to the case of a  general closed convex front track, the body  curve ${\bf N}$ on $H^{1,1}$ is obtained by deforming the equator ${\bf n}$  by changing  its  geodesic curvature from $0$ to $ 1/(\ell \kappa)$; the resulting deformation ``splits" what initially was the closed curve, with the endpoints and the tangents at the endpoints related by
\[
	 \Ad_{g(L)}  {\bf N} (L)= {\bf N}(0), \  \  \Ad_{g(L)} \dot {\bf N} (L)= \dot{\bf N}(0), 
\]  
as Figure \ref{fig:rollingellipse} illustrates. It turns out that the split is rather special for large $ \ell $:   the endpoints separate almost tangentially, as Figure \ref{fig:rollingellipse} suggests, and {\it  the distance of  
separation is proportional to the area $A$ enclosed by the front track}, to the leading order, as Figure~\ref{fig:rollingellipse} suggests. Indeed, this follows from the following observation. 
\begin{lemma}\label{lem:Ad} Let $A$ be the area enclosed by the front track $F$. 
  For large $ \ell$, the adjoint action $\Ad_{g(L)}$ is an elliptic rotation through an angle 
\begin{equation} 
	 \ell ^{-2}A+ O(\ell^{-3} ),
	\label{eq:rotationangle}
\end{equation}   around a timelike axis which is $O( \ell ^{-1}) $ -- close to the $ a_3 $ axis in $ \R^{2,1}$. 
\end{lemma}
 In the special case when the front track is the unit circle, the picture is particularly simple, Figure~\ref{fig:rollingellipse}: 
the body curve ${\bf N}$  is an arc of an ellipse lying in a plane tangent to the equator and of slope $ \ell ^{-1}$ (exactly); and the axis of the rotation $\Ad_{g(L)}$  is the line of slope $ \ell $ (in the Lorenz plane the orthogonal lines have reciprocal slopes; in other words, the slope of the light line is the geometric mean of two orthogonal slopes). 
 
\mn {\em Proof Lemma \ref{lem:Ad}.} 
\begin{enumerate}[1.]
\item As stated before,  we  assume $ F(t) $ to be a closed front track and $ \ell $ to be large.  According to Prytz's  
formula (see \cite{F} or  \cite[equation (1)]{BLPT}) the bicycle angle $\theta$ governed by   (\ref{eq:bike}) changes, after the front wheel traces out the front track, by   
\begin{equation} 
  	     \Delta\theta = \ell ^{-2} A + O(\ell ^{-3}  ). 
	    \label{eq:prytz}
\end{equation}  
   In particular, the rotation is near--rigid: the leading order term is independent on the initial condition $ \theta (0) $. 
\item   According to    (\ref{eq:2arg}), every solution  $ (u,v) $ of   (\ref{eq:bikesys}) rotates through half as much as $\theta$ does:   
\[
	\Delta \arg (u+iv)  = \frac{1}{2} \Delta\theta \ \stackrel{\eqref{eq:prytz}}{=}  \ \frac{1}{2} 	\ell ^{-2} A + O(\ell ^{-3}  );
\]
and since these angles are independent of the initial condition modulo $ \ell^{-3} $, we conclude that $ g(L) $ 
is $ O(  \ell ^{-3}) $--close to the Euclidean rotation through $  \frac{1}{2} 	\ell ^{-2} A $.   And this in turn implies
that $ \Ad_{g(L)}$ is $ O(  \ell ^{-3}) $--close to the Euclidean=Minkowski rotation $ R $ around the ${\bf k}$--axis in the Minkowski space 
through twice the angle, namely through
\[ \ell ^{-2} A + O(\ell ^{-3}). 
\] 
\item   This proximity  in turn implies via an implicit function argument 
that the  the Minkowski rotation axis  of $  \Ad_{g(L)}$ (i.e. the eigendirection corresponding to the eigenvalue $1$) is 
 $ O(\ell ^{-1} ) $--close to the ${\bf k}$--axis. Indeed, consider the maps induced by the linear maps $\Ad_{g(L)}$ and $R$ on the unit sphere, and examine what happens to the fixed point ${\bf k}$ of  $R$ as we perturb 
 $R$ to  $ \Ad_{g(L)}$. By an    implicit function argument, the displacement of the  fixed point is bounded by the size of the perturbation 
 ($ O(\ell ^{-1} ) $) 
 divided by the distance from $R$ to identity, which is {\it  at least} $\frac{1}{2}\ell^{-2} A $; thus the fixed point is displaced by {\it  at most}
 \[
	\frac{O(\ell ^{-3} )}{ \frac{1}{2} A\ell ^{-2} } =O(\ell ^{-1} ).
\]

 \item Finally, by the Minkowski orthogonality, the invariant plane
of $ \Ad_{g(L)}$ corresponding to the eigenvalues $ \pm i \bigl(\ell ^{-2} A +O(\ell ^{-3}) \bigr)  $ has the reciprocal slope, i.e., this plane  is  $ O(\ell^{-3} ) $--close to the equatorial plane. \qed  
 \end{enumerate}

\end{document}

%% file: preamble.tex
\usepackage{makecell}
\usepackage{mathrsfs}
\setcellgapes{4pt}

\newcommand{\centered}[1]{\begin{tabular}{l} #1 \end{tabular}}

\newcommand{\bfn}{{\mathbf n}}

\newcommand{\bfN}{{\mathbf N}}
\usepackage{import}
\usepackage{xifthen}
\usepackage{pdfpages}
\usepackage{transparent}

\usepackage{epstopdf}

\newcommand{\bi}{\mathbf{i}}
\newcommand{\bj}{\mathbf{j}}
\newcommand{\bk}{\mathbf{k}}
\newcommand{\bod}{\mathbf{N}}
\newcommand{\spa}{\mathbf{n}}

\newcommand{\Cs}{\mathscr{C}_{space}}
\newcommand{\Cb}{\mathscr{C}_{body}}

\newcommand{\TT}{T}

\usepackage{mathrsfs}
\usepackage{amsmath} 
\usepackage{amsfonts}
\usepackage{amsthm}
\usepackage{url}
\usepackage{graphicx}		
\usepackage{epsfig,psfrag}  
\usepackage{enumerate}
\newtheorem{lemma}{Lemma}[section]

\newtheorem{theorem}{Theorem}

\theoremstyle{definition}
\newtheorem{definition}[lemma]{Definition}
\newtheorem{rmrk}[lemma]{Remark}

\newcommand{\e}{{\mathbf e}}

\newcommand{\SOt}{\mathrm{SO}_3}

\renewcommand{\v}{\mathbf{v}}

\newcommand{\II}{\mathrm I}

\renewcommand{\>}{\rangle}
\newcommand{\<}{\langle}

\newcommand{\x}{{\mathbf x}}

\newcommand{\tr}{\mbox{tr}}

\renewcommand{\u}{{\bf u}} 
\newcommand{\g}{\mathfrak{g}}

\renewcommand{\a}{{\bf a}}

\renewcommand{\d}{{\bf d}}
\renewcommand{\o}{ \boldsymbol{\omega}}

\newcommand{\st}{\, | \,}

\newcommand{\R}{\mathbb{R}}

\newcommand{\Z}{\mathbb{Z}}

\newcommand{\SL}{\mathrm{SL}}

\newcommand{\SLt}{{\SL_2(\R)}}

\newcommand{\SO}{\mathrm{SO}}

\newcommand{\slt}{\mathfrak{sl}_2(\R)} 
 
\renewcommand{\sl}{\mathfrak{sl}}

\newcommand{\so}{\mathfrak{so}} 
\newcommand{\sot}{{\so_3}}

\newcommand{\Ad}{{\rm Ad}} 
\renewcommand{\mod}{{\rm mod\; }}

\newcommand{\n}{\noindent} 
\newcommand{\bs}{\bigskip}
 
\newcommand{\sn}{\smallskip\n} 

\newcommand{\mn}{\medskip\noindent}
\newcommand{\bn}{\bs\n}

\newcommand{\be}{\begin{equation}}
\newcommand{\ee}{\end{equation}}

\usepackage[all]{xy}

\renewcommand{\d}{\mathrm{d}}

\usepackage{color}

\usepackage{bm}

\newcommand{\ob}{{\boldsymbol{\omega}}}
\newcommand{\Ob}{\boldsymbol{\Omega}}



%% file: main.bbl
\begin{thebibliography}{9}

\bibitem{Ar} V.I.~Arnol'd, {\em Mathematical methods of classical mechanics.} Vol. 60. Springer Science \& Business Media, 2013.

\bibitem{BLPT} G.~Bor, M.~Levi, R.~Perline, S.~Tabachnikov, {\em Tire Tracks and Integrable Curve Evolution,} Int. Math. Res. Not. IMRN (2018)

\bibitem{F} R.~Foote,  {\em Geometry of the Prytz planimeter,} Rep. Math. Phys. 42 (1998), 249--71.

\bibitem{Ma} M.~Levi, {\em Composition of rotations and parallel transport}.
Nonlinearity {\bf 9.2} (1996), 413.

\bibitem{Po}L.~Poinsot, {\em  Th\'eorie nouvelle de la rotation des corps}. Bachelier (1854).

\bibitem{Wh} E.T.~Whittaker, {\em  A treatise on the analytical dynamics of particles and rigid bodies}. Cambridge University Press, 2nd edition (1917).




 \end{thebibliography}
